\documentclass[11pt]{amsart}
\usepackage[babel]{csquotes}
\usepackage{enumitem}
\usepackage{amsmath,amsthm,amssymb,mathrsfs,amsfonts,verbatim,enumitem,color,leftidx}
\usepackage{mathabx}
\usepackage{etoolbox} 
\usepackage{bbm}
\usepackage[all,tips]{xy}
\usepackage{graphicx,ifpdf}
\usepackage{stmaryrd}
\ifpdf
\fi
\usepackage[colorlinks]{hyperref}
\hypersetup{
linkcolor=blue,          
citecolor=green,        
}

\usepackage{tikz}

\newtheorem{thm}{Theorem}[section]
\newtheorem{lem}[thm]{Lemma}
\newtheorem{cor}[thm]{Corollary}
\newtheorem{prop}[thm]{Proposition}

\theoremstyle{definition}
\newtheorem{defi}[thm]{Definition}

\newtheorem{xca}[thm]{Exercise}

\theoremstyle{remark}
\newtheorem{rem}[thm]{Remark}

\numberwithin{equation}{section}
\definecolor{esperance}{rgb}{0.0,0.5,0.0}

\newcommand{\seon}[1]{\textsc{} {\color{magenta}\marginpar{} \textsf{#1}}}

\newcommand{\bbk}{\mathbf{k}}

\newcommand{\R}{\mathbb{R}}

\newcommand{\Z}{\mathbb{Z}}

\newcommand{\s}{\sigma}
\newcommand{\de}{\delta}

\newcommand{\Ybe}{F_{S^+(\cO)}}
\newcommand{\Ybr}{E^\cO_{[r,\infty)}}
\newcommand{\Ybm}{\E{r_m}}
\newcommand{\Ybmprime}{E^{\cO'}_{[r_m,\infty)}}

\DeclareMathOperator{\diam}{diam}
\DeclareMathOperator{\card}{card}
\DeclareMathOperator{\Grass}{Grass}
\DeclareMathOperator{\Span}{Span}

\newcommand{\Bad}{\mb{Bad}}
\newcommand{\Badg}{\mb{Bad}_{\ell,d}}

\newcommand{\al}{\alpha}

\newcommand{\del}{\delta}

\newcommand{\lam}{\lambda}

\newcommand{\eps}{\epsilon}

\newcommand{\sig}{\sigma}

\newcommand{\om}{\omega}
\newcommand{\Om}{\Omega}

\newcommand{\cA}{\mathcal{A}}
\newcommand{\cB}{\mathcal{B}}
\newcommand{\cC}{\mathcal{C}}

\newcommand{\cE}{\mathcal{E}}

\newcommand{\cH}{\mathcal{H}}

\newcommand{\cO}{\mathcal{O}}
\newcommand{\cP}{\mathcal{P}}

\newcommand{\bR}{\mathbb{R}}
\newcommand{\bZ}{\mathbb{Z}}

\newcommand{\bN}{\mathbb{N}}

\newcommand{\SL}{\operatorname{SL}}

\newcommand{\ASL}{\operatorname{ASL}}

\newcommand\norm[1]{\|#1\|}

\newcommand\set[1]{\left\{#1\right\}}
\newcommand\pa[1]{\left(#1\right)}
\newcommand\idist[1]{\langle#1\rangle}

\newcommand\av[1]{|#1|}
\newcommand\on[1]{\operatorname{#1}}
\newcommand\diag[1]{\operatorname{diag}\left(#1\right)}
\newcommand\mb[1]{\mathbf{#1}}
\newcommand\E[1]{E^\cO_{[#1,\infty)}}

\newcommand\tb[1]{\textbf{#1}}
\newcommand\mat[1]{\pa{\begin{matrix}#1\end{matrix}}}

\newcommand\smallmat[1]{\pa{\begin{smallmatrix}#1\end{smallmatrix}}}
\newcommand\crly[1]{\mathscr{#1}}

\newcommand{\wstar}{\overset{\on{w}^*}{\lra}}
\newcommand{\Supp}{\on{Supp}}

\newcommand{\defn}{\overset{\on{def}}{=}}

\newcommand{\lra}{\longrightarrow}

\newcommand{\onto}{\xymatrix{\ar@{>>}[r]&}}

\newcommand{\dima}{\underline{\dim}_a}
\newcommand{\dimm}{\underline{\dim}_M}

\newcommand{\eqlabel}[2]
{
\begin{equation}
{#2}\label{#1}
\end{equation}
}

\newcommand{\seonnote}[1]{\marginpar{\color{magenta}\tiny [SS] #1}}

\begin{document}

\title{Dimension bound for badly approximable grids}

\author{Seonhee Lim}
\author{Nicolas de Saxc\'e}
\author{Uri Shapira}

\thanks{}



\date{}


\begin{abstract}
We show that for almost any vector $v$ in $\bR^n$, for any $\eps>0$ there exists $\del>0$ such 
that the dimension of the set of vectors $w$ satisfying $\lim\inf_{k\to\infty} k^{1/n}\idist{kv-w}\ge \eps$
(where $\idist{\cdot}$ denotes the distance from the nearest integer), is bounded above by $n-\del$. 
This result is obtained as a corollary of a discussion in homogeneous dynamics and the main tool 
in the proof is a relative version of the principle of uniqueness of measures with maximal entropy.
\end{abstract}

\maketitle
\section{The main result and its applications}

\subsection{Geometry of numbers}
A general theme in the geometry of numbers is to fix a domain $S\subset \bR^d$ and study the intersection of it with sets possessing
an
algebraic structure such as lattices or their cosets.
One is usually interested in bounding the cardinality of such an intersection and 
this will be the case in our discussion as well. We begin by describing the domains we will consider and which we refer to as \textit{spikes}. Throughout, we fix a dimension $d\ge 2$ and a diagonal flow 
$$a_t=\diag{e^{c_1t},\dots,e^{c_dt}};$$
where $c_i$ are fixed \emph{non-zero} numbers such that $\sum c_j = 0.$ 
Given a bounded open set $\cO\subset \bR^d$, we define the \textit{positive\footnote{One could work with two-sided spikes taking the union over $t\in \bR$ in~\eqref{eq1608} but our results are stronger as the 
domain decreases so we will concentrate on the one-sided case.}
 spike of $\cO$ with respect to $a_t$} to be the set 

\begin{equation}\label{eq1608}
S^+(a_t, \cO) = S^+(\cO) \defn \bigcup_{t>0} a_t^{-1}\cO.
\end{equation} 

 \begin{figure}
\begin{center}
\includegraphics[width=4cm]{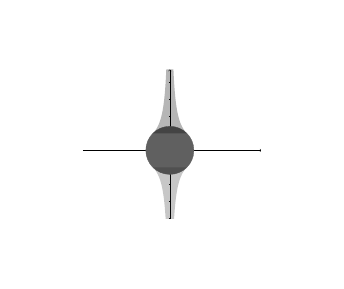}
\includegraphics[width=3.3cm]{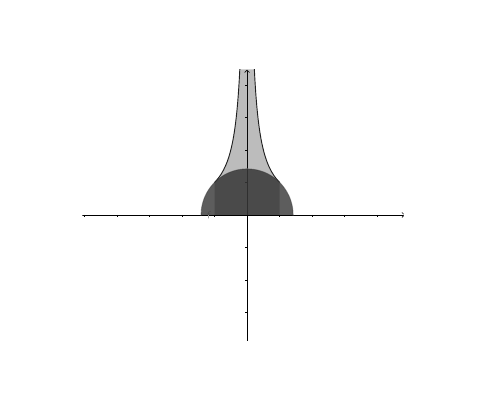}
\includegraphics[width=5cm]{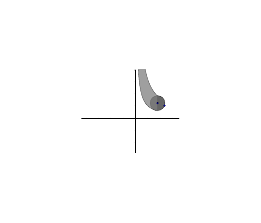}
\caption{Spikes of a ball around 0, half-ball around 0, a ball around (3,2) for $a_t=\diag{e^t, e^{-t}}$}
\end{center}
\end{figure}
As $a_t$ will be fixed throughout our discussion we omit it from the notation. 
The space of unimodular lattices (i.e.\ of covolume 1) in $\bR^d$ will be denoted by $X$. 
By a unimodular \emph{grid} $y$ in $\R^d$, we mean a coset $x+\mb{w}$ of a lattice $x\in X$ where $\mb{w}\in\bR^d$. 
We denote by $Y$ the space of unimodular grids in $\bR^d$ and by $\pi:Y\to X$ the natural projection.
Note that for $x\in X$, the fiber $\pi^{-1}(x)$ is simply the torus $\bR^d/x$.
For $x\in X$ and an open set $\cO\subset \bR^d$, we set
\begin{align}\label{eq:1136}
&\nonumber \Ybe\defn\set{y\in Y:y\cap S^+(\cO)\ \mbox{is finite}};\\
&\Ybe(x)\defn \Ybe\cap \pi^{-1}(x).
\end{align}  

Our main result, Theorem~\ref{thm:main} below, says that under a mild dynamical assumption on the forward $a_t$-orbit of $x$, the set $\Ybe(x)$ cannot have maximal Hausdorff dimension in $\pi^{-1}(x)$.
Let us introduce this dynamical assumption.
The standard action of $G_0\defn\SL_d(\bR)$ on $\R^d$ induces a transitive action of $G_0$ on  $X$, and since for $x_0=\bZ^d$ we have $\on{Stab} x_0=G_0(\Z)$, we may identify $X\simeq G_0/G_0(\bZ)$.
This endows $X$ with a smooth manifold structure and with a unique $G_0$-invariant Borel probability measure, which we denote by $m_X$. 

For a locally compact second countable Hausdorff space $Z$ we will denote by $\crly{P}(Z)$ the space of Borel probability measures on $Z$ and endow it with the weak* topology
by identifying $\crly{P}(Z)$ with a subset of the unit sphere in the dual of $C_0(Z)$. 
For $\mu\in \crly{P}(Z)$ and $f\in C_0(X)$ we alternate between the notations $\mu(f)$ and $\int fd\mu$.
We will denote by $\del_z\in \crly{P}(Z)$ the Dirac probability measure at $z$. If a map $g:Z\to Z$ is fixed we 
denote for $z\in Z$ and $T\in \bZ_+$, 
$\del_z^T\defn \frac{1}{T}\sum_{i=0}^{T-1}\del_{g^iz}\in \crly{P}(Z)$. Note that by the Banach-Alaoglu theorem
$\set{\al \mu: \al\in[0,1], \mu\in \crly{P}(Z)}$ is compact in the weak* topology and thus for any $z\in Z$, the sequence 
$\del_z^T\in \crly{P}(Z)$ has accumulation points of the form $\al\mu$ with $\al\in[0,1]$ and $\mu\in \crly{P}(Z)$. The following is concerned with the situation
where $\del_z^T$ can accumulate on a \emph{probability} measure.
Throughout we use the notation $\del_x^T$ for the transformation $a=a_1:X\to X$.
\begin{defi}[Heavy lattice]\label{def:heavy lattice}
\quad 
\begin{enumerate}
\item A lattice $x\in X$ is called \textit{heavy} (for $a_t$ in positive time) 
if $$\set{\del_x^T: T\in\bZ_+}^\om \cap\crly{P}(X)\ne\varnothing,$$  
where $F^\om$ is the set of accumulation points of $F$.
\item We fix once and for all a sequence of compactly supported functions $\psi_i\in C_c(X)$ such that
$0\le \psi_i\le 1$, and $\psi_i^{-1}(1)$ is an increasing sequence of compact sets that covers $X$. 
Given a sequence of non-negative numbers
$\eta_i\to 0$, we  define 
$$\crly{P}(X, (\eta_i))\defn\set{\mu\in \crly{P}(X): \forall i,\; \mu(\psi_i)\ge 1-\eta_i}.$$
\item Given a sequence of non-negative numbers $\eta_i\to 0$,  we define 
$$\cH(\eta_i)\defn\set{x\in X: \set{\del_x^T:T\in\bZ_+}^\om \cap  \crly{P}(X, (\eta_i))\ne \varnothing}$$  
\end{enumerate}
\end{defi}
As the following lemma shows, any heavy lattice belongs to some $\cH(\eta_i)$. 
The point in defining $\crly{P}(X,(\eta_i))$ and $\cH(\eta_i)$ is that our results about heavy lattices 
will be uniform on $\cH(\eta_i)$.
\begin{lem}\label{rem:1511}\quad\\
\begin{enumerate}
\item\label{r01} $\crly{P}(X)=\bigcup \crly{P}(X,(\eta_i))$ where the union is taken over all
sequences of non-negative numbers  $\eta_i\to 0$. 
\item\label{r012} The set of heavy lattices equals $\bigcup \cH(\eta_i)$ where the union is taken over all
sequences of non-negative numbers  $\eta_i\to 0$. 
\item\label{r02} $\crly{P}(X, (\eta_i))$ is compact .
\end{enumerate}
\end{lem} 
\begin{proof}
\eqref{r01}. 
For $\mu\in \crly{P}(X)$, $\mu\in \crly{P}(X,(\eta_i))$
for $\eta_i \defn 1- \mu(\psi_i)$. Note that $\eta_i\to 0$ because $\psi_i^{-1}(1)$ is an increasing cover of $X$.\\
\eqref{r012}.
Let $x$ be a heavy lattice and let $\mu\in \set{\del_x^T:T\in \bZ^+}^\om\cap \crly{P}(X)$. By part \eqref{r01} $\mu\in \crly{P}(X,(\eta_i))$ for some sequence $(\eta_i)$. By definition we then have that $x\in \cH(\eta_i)$.\\
\eqref{r02}. Let $\mu_m\in\crly{P}(X,(\eta_i))$ be a sequence and let $\mu$ be a weak* accumulation point of it. 
For each $i$ we have $\mu(X)\ge \mu(\psi_i)=\lim_m\mu_m(\psi_i)\ge 1-\eta_i$. Letting $i\to\infty$ we obtain
$\mu(X)=1$ and $\mu(\psi_i)\ge 1-\eta_i$ so that $\mu\in \crly{P}(X,(\eta_i))$ by definition.
\end{proof}

Our main result is as follows. Here, $\dim_H$ denotes Hausdorff dimension with respect to the Euclidean metric on $\pi^{-1}(x) \simeq \R^d/x$.

\begin{thm}[Heavy lattices have few bad grids]\label{thm:main}
For any bounded open set $\cO\subset \bR^d$, if $x$ is heavy (for $a_t$ in positive time), then $$\dim_H \Ybe(x)<d.$$
In fact, for a given $\eta_i\to 0$, there exists $\del = \del(\cO, (\eta_i))>0$ such that for any $x\in\cH(\eta_i)$, 
$$\dim_H \Ybe(x)<d-\del.$$
\end{thm}
By Lemma~\ref{rem:1511}\eqref{r012}, results stated for lattices in $\cH(\eta_i)$ for arbitrary 
$(\eta_i)$ automatically hold for heavy lattices. Thus, the second part of Theorem~\ref{thm:main} implies
the first and demonstrates the uniformity gained by exhausting the set of heavy lattices by the sets $\cH(\eta_i)$.
The following corollary shows that this uniformity survives if one is only interested in an almost sure statement with respect 
to the smooth measure $m_X$. 

\begin{cor}[Random lattices have few bad grids]\label{cor:main}
For any bounded open set $\cO\subset \bR^d$, there exists $\del>0$ such that for $m_X$-almost any lattice $x$,
$\dim_H \Ybe(x)\le d-\del.$
\end{cor}
\begin{proof}
By the ergodicity of the action of $a_1$ on $X$, for almost any $x$, $\del_x^T\wstar m_X$ which 
trivially implies  $m_X\in \set{\del_x^T:T>0}^\om$. By Lemma~\ref{rem:1511}\eqref{r01}, $m_X\in \crly{P}(X,(\eta_i))$ for a
suitable sequence $(\eta_i)$. Thus by definition $x\in \cH(\eta_i)$. 
The result then follows from Theorem~\ref{thm:main}. 
\end{proof}

\subsection{An application to Diophantine approximation}
\label{subsec:app}

For a vector $v\in\bR^n,$ we are interested in the behaviour of the sequence 
$$\set{kv \on{mod} \bZ^n\ ;\ k\in\bN}\subset \bR^n/\bZ^n.$$ 
If $v$ does not belong to a rational subspace, this sequence is dense and even equidistributed so that for any \textit{target} $w\in \bR^n$, we have that $\inf_{k\ge 1}\idist{kv-w}=0$, where $\idist{t}$ denotes the distance from $t$ to $\bZ^n$.  
A more subtle question is whether 
$\liminf_{k\to\infty}\psi(k)\idist{kv-w}=0$ for some prescribed function $\psi\nearrow\infty$ on $\bN$. One may visualize this as a \textit{shrinking target} problem where one asks if for any $\eps>0$ and for arbitrarily large $k$, 
the point $kv \on{mod} \bZ^n$ on the $n$-torus is inside the ball of radius $\psi(k)^{-1}\eps$ centered at $w$ (which is the shrinking target).  
The most classical choice, and the one that we will consider, is $\psi(k) = k^{1/n}$. 
We call $w$ $\eps$-\textit{bad} for $v$ if
\eqlabel{eq1523}{
\liminf_{k\to \infty} k^{1/n}\idist{kv-w}\ge \eps,
}
and denote 
\begin{align*}
\mb{Bad}^\eps(v)&\defn\set{w\in\bR^n:w\textrm{ is $\eps$-bad for $v$}},\\ 
\mb{Bad}(v)&\defn\bigcup_{\eps>0}\mb{Bad}^\eps(v).
\end{align*}
Our main application is the following.
\begin{thm}\label{thm:app}
For any $\eps>0$ there exists $\del>0$ such that for Lebesgue almost every $v\in \R^n$,
$\dim_H\mb{Bad}^\eps(v) < n-\del$.
\end{thm}

To put this in context we mention that
it follows from~\cite{BHKV10} that for any $v\in\bR^n$, $\dim_H \mb{Bad}(v)=n$. Later, it was shown in \cite{ET} that $\mb{Bad}(v)$ is a winning set.
Note also that the conclusion of the theorem cannot hold for every $v\in\R^n$.
Indeed, if $v$ lies in a rational subspace (and hence is 
trivially singular), then for small enough $\eps$, the set $ \mb{Bad}^\eps(v)$ has non-empty interior and thus obviously has dimension $n$. 
Moreover, in Section~\ref{sec:ex} we construct non-singular vectors
which violate the conclusion of the theorem and satisfy $\dim_H \mb{Bad}^\eps(v)=n$ for a positive $\eps$. 
As we will see in Section~\ref{sec:mt}, Theorem~\ref{thm:app} holds not only for almost every $v$, but for any heavy vector $v$ (although the constant $\del$ in Theorem~\ref{thm:app}  might depend on $v$), see Definition~\ref{def:heavy vector} and Theorem~\ref{thm:appug}.
It follows from~\cite{GDV} that for non-singular vectors $\lam( \mb{Bad}(v)) = 0$, where $\lam$ is the Lebesgue measure on $\bR^n$. Thus the above theorem is an
upgrade of the result in \cite{GDV} just mentioned under the stronger assumption of heaviness. 
We refer the reader to Section~\ref{sec:mt} for other examples of applications of Theorem~\ref{thm:main}, such as diophantine approximation of affine subspaces of $\R^n$.


\subsection{Outline of the proof of Theorem~\ref{thm:main}}

We briefly describe our strategy which is similar in spirit to the idea given in~\cite[Remark 2.2]{BM} and which could be described in a 
nutshell as rigidity of measures with maximal entropy. 
Assuming by way of contradiction that $\dim \Ybe(x) =d$ for a heavy lattice $x$, we construct a sequence of 
probability measures $\mu_k$ defined by taking the uniform measures $\nu_k$ 
on large finite sets $S_k\subset \Ybe(x) \subset \pi^{-1}(x)$ and averaging them along the $a_1$-orbit; $\mu_k = n_k^{-1}\sum_{i=1}^{n_k}(a_1)^i_*\nu_k,$ 
for a suitable $n_k$ (defined in the proof of Proposition~\ref{largeentropy}).
The heaviness assumption allows to take a weak-* limit $\mu$ of $\mu_k$ that is a probability measure on $Y$ and moreover, the maximal dimension assumption translates into the maximality 
of the relative entropy of $\mu$ with respect to 
$a_1$ relative to the factor $X$.
We then prove that maximality of the relative entropy implies invariance of $\mu$ under the whole subgroup of translations in $\bR^d$.
This leads us to a contradiction because by construction, $\mu$ is supported on the accumulation points of forward $a_t$-orbits of the points in $\Ybe$ and in particular, must be supported in the closed $a_t$-invariant set 
$\set{y\in Y:\ 0\in y\ \mbox{or}\ \forall t\in\R,\, a_t\cdot y\cap \cO=\varnothing}$.

%

\subsection{Plan of the paper}

Apart from this introduction, this paper consists of five parts. In Section~\ref{sec:large-entropy}, we show that a set of large dimension in the space of grids, contained in a single fiber, can be used to construct an $a_t$-invariant measure on $Y$ with large entropy with respect to the factor $X$.
Then, in Section~\ref{sec:invariance}, we study $a_t$-invariant measures with maximal relative entropy with respect to $X$; using work of Einsiedler and Lindenstrauss, we show that they are invariant under the unstable horospherical  
subgroup $U^+$ associated to $a_t$. 
In Section~\ref{sec:mt}, using the two previous sections, we finish the proof of Theorem~\ref{thm:main}, before explaining the applications to Diophantine approximation in Section~\ref{sec:diophantine}.
We conclude the paper with Section~\ref{sec:ex}, detailing the construction of some lattices with non-divergent $a_t$-orbits 
that do not satisfy the conclusion of Theorem~\ref{thm:main}.

\section{Measures with large entropy}
\label{sec:large-entropy}

Given a heavy lattice $x$ for $a_t$ and a set $S\subset\pi^{-1}(x)$ of grids lying above $x$, we explain how to construct a measure $\mu$ on $Y$ with large entropy relative to the factor $X$, which is supported in the closure of the forward orbit of $S$ under the diagonal flow $a_t$.

\subsection{Action on the space of grids}

Just like for the space $X$ of lattices in $\R^d$, one can view the space $Y$ of unimodular grids as a homogeneous space.
Indeed, the natural action on $\R^d$ of the group $G=\ASL_d(\R)$ of area-preserving affine transformations induces a transitive action on $Y$, with $\on{Stab}y_0=G(\Z)$ if $y_0=\Z^d$, so that $Y\simeq G/G(\Z)$ has a natural smooth manifold structure, and carries a unique $G$-invariant probability measure $m_Y$.

We denote by $U$ the unipotent radical of $G$, which consists of all translations on $\R^d$. It is clear that $U$ acts simply transitively on $\R^d$.
One naturally identifies $G_0=\SL_d(\R)$ with $G/U$, and then, the canonical projection $\pi:Y\to X$ intertwines the actions of $G$ and $G_0$ on $Y$ and $X$, respectively, in the sense that for any $g\in G$ and $y\in Y$,
\[ \pi(g\cdot y) = \bar{g}\cdot\pi(y),\]
where $\bar{g}$ denotes the projection of $g$ to $G_0\simeq G/U$.
Clearly, if $y$ is a grid with underlying lattice $x$, then
\[ \pi^{-1}(x)=Uy.\]
It is sometimes convenient to view $G$ as a subgroup of $\SL_{d+1}(\R)$ by
\[ G=\{
\left(\begin{matrix}
A & u\\
0 & 1\\
\end{matrix}\right);\ 
A\in \SL_d(\R),\ u\in\R^d\},\]
in which case the unipotent radical is $U = \set{\smallmat{I_d & u\\ 0 &1}:u\in\bR^d}$.
\bigskip

Recall that $a_t=\diag{e^{c_1t},\dots,e^{c_dt}}$ is a one-parameter diagonal subgroup in $G_0=\SL_d(\R)$.
We take a lift of this one-parameter group to $G\subset \SL_{d+1}(\bR)$ given by $\smallmat{a_t&0\\0&1}$
and by abuse of notation we denote it again by $a_t$.
It will be convenient to normalize the flow $a_t$ so that 
\eqlabel{cnormalize}{
\max_{1\leq i\leq d} c_i=1.
}
We let $a\defn a_1$ be the time-one map for the diagonal flow and denote by $G^+$ the unstable horospherical subgroup for $a$ in $G$, i.e. 
\[ G^+ \defn \{ g\in G\ |\ a_tga_t^{-1} \to e\ \mbox{as}\ t\to-\infty\}.\]
We let $U^+\defn U\cap G^+$ so that
if $J^+=\{i\in\{1,\dots,d\}\,|\, c_i>0\}$, then we have
\[ U^+=\{
\left(\begin{matrix}
I_d & u\\
0 & 1\\
\end{matrix}\right);\ 
u=\, ^t(u_1,\dots,u_d)\in\R^d\ \mbox{with}\ u_i=0\ \forall i\not\in J^+\},\]

On $U$, we will use the Euclidean distance $d_E$ inherited from $\R^d$. This metric induces a metric, still denoted by $d_E$, on the fiber $\pi^{-1}(x)\simeq\R^d/x$ of all grids lying above $x \in X$.
Given a grid $y$, we define the \textit{fiber injectivity radius} at $y$ to be the maximal number $r_y>0$ such that the orbit 
map $u\mapsto uy$ is injective, and therefore isometric, on the open ball $B_{r_y}^{U,d_E}(0)$ of radius $r_y$ in $U$ for the Euclidean metric $d_E$.
Note that the fiber injectivity radius is constant along the fiber and is bounded away from zero on compact sets in $Y$. 

On $U^+$, we will also make use of another metric, or rather, quasi-metric, more adapted to the action of $a_t$. 
We define the quasi-norm associated to $a$ by  $|u|_a=\max_{j\in J^+}|u_j|^{1/c_j}$.
The function on $U^+\times U^+$, given by $d_a(u,v)=|u-v|_a$ is a quasi-metric: it is symmetric, positive definite, and satisfies, for some constant $C$ depending on the $c_i$, for all $u,v,w$ in $U^+$, $d_a(u,w)\leq C(d_a(u,v)+d_a(v,w))$.
The ball $B_\delta^{U^+,d_a}(u)$ of radius $\delta$ around $u$ for $d_a$ is simply the set of $v\in U^+$ such that $d_a(u,v)< \delta$.

\begin{rem}\label{remark1}We observe two things:
\begin{enumerate}
\item\label{1110} A ball $B^{U^+,d_a}_\delta$ in $U^+$ is simply a box with side-lengths $2\delta^{c_j}$, $j\in J^+$ with respect to $d_E$.
\item\label{1111} The action of $a_t$ on $U^+$ is a dilation by a factor of $e^t$ for the quasi-metric $d_a$; that is, for all $u,v$ in $U^+$ and $t\in\R$, we have 
\eqlabel{eq: scaling}{
d_a(a_tu,a_tv)=e^td_a(u,v),
}
\end{enumerate} 
\end{rem}
We let $W_y^+$ be the image of  $B_{r_y}^{U^+,d_E}$ under the orbit map. 
We call $W_y^+$  \emph{the injective unstable leaf at $y$ in the fiber}.
By definition of $r_y$, the orbit map identifies $(B_{r_y}^{U^+,d_E}, d_E)$ and $(W_y^+, d_E)$ isometrically.
Of course, the quasi-distance $d_a$ on $U^+$ also induces a quasi-distance on $W_y^+$, which we again denote by $d_a$. 
Since we will use both distances $d_E$ and $d_a$ on $W_y^+$ (which are far from being equivalent if $c_j<1$  for some $j\in J^+$), 
we will indicate the metric in the superscript when necessary.

\subsection{Dimensions}

Let $X$ be a space endowed with a quasi-distance $d$.
For a bounded subset $S\subset X$ we will consider its lower Minkowski dimension (or lower box dimension)
$\underline{\dim}_{d} S$ for the quasi-metric $d$, defined by
\[ \underline{\dim}_d S \defn \liminf_{\delta\to 0} \frac{\log N_d(S,\delta)}{\log\frac{1}{\delta}},\]
where $N_d(S,\delta)$ is the maximal cardinality of a $\delta$-separated subset of $S$ for the quasi-metric $d$.
If $S$ is unbounded, we let $\underline{\dim}_d S = \sup\{\underline{\dim}_d S\cap K\ ;\ K\ \mbox{compact}\}$.

In particular, for a set $S\subset W_y^+$, we will consider its lower Minkowski dimensions 
$$\dima S \defn \underline{\dim}_{d_a} S, \qquad \underline{\dim}_M S \defn \underline{\dim}_{d_E}S$$ for the quasi-metric $d_a$ and the Euclidean metric $d_E$, respectively.
We will also consider the Hausdorff dimension $\dim_H S$, always defined with respect to the Euclidean metric.
We refer the reader to \cite{falconer} for general properties of Minkowski or Hausdorff dimensions, such as the inequality
\[ \dimm S\geq \dim_HS.\]

We introduce $\dima$ in order to relate dimension $\underline{\dim}_M$ to entropy, and further to Hausdorff dimension. The following simple observation gives a relation between  $\underline{\dim}_M S$ and  $\underline{\dim}_a S$. Let 
$$h_a= \sum_{i\in J^+} c_j.$$

\begin{lem}\label{relating dimensions}
We have $\dima U^+=h_a$, and moreover, for any set $S\subset U^+$,  $$\dima S \ge \underline{\dim}_M S+h_a-\dim U^+.$$
\end{lem}
\begin{proof}
A $\delta$-ball for $d_a$ is a Euclidean box with side lengths $2\del^{c_i}$, so that any bounded set in $U^+$ can be covered by $O(\delta^{-h_a})$ balls of radius $\delta$ for $d_a$. Conversely, for volume reasons, one needs at least $O(\delta^{-h_a})$ $\delta$-balls for $d_a$ to cover any non-empty open set in $U^+$. This shows the first equality.

For the general inequality, let $u = \dim U^+$. 
Each $d_a$-ball of radius $\delta$ can be covered by at most 
$\delta^{h_a-u}$ boxes of side lengths 2 $\delta$ which in turn can be covered by the same number of Euclidean balls (up to a multiplicative constant, say $C$).
Thus $N_{d_a}(S,\delta) \geq C \delta^{h_a-u}N_{d_E}(S,\delta)$. Taking logarithms, dividing by $\log 1/\de$, and taking 
$\del\to 0$ gives the result.
\end{proof}

\subsection{Constructing a measure of large entropy}

We refer the reader to \cite[\S 2.2]{ELW} for the definition of relative entropy with respect to an infinite countably generated $\sigma$-algebra ; in particular, if $\cP$ is any countable partition of $Y$, then $H_{\mu}(\cP|X)$ will denote the relative entropy of $\cP$ with respect to the $\sigma$-algebra $\pi^{-1}(\cB_X)$, i.e.\ the inverse image under $\pi$ of the Borel $\sigma$-algebra $\cB_X$ on $X$.
Finally, $h_\mu(a|X)$ denotes the relative entropy of the transformation $a=a_1:Y\to Y$ for the measure $\mu$ (relative to $X$), i.e.
\[ h_\mu(a|X) 
\defn \sup_{\cP} \inf_{q\in\Z^+} \frac{1}{q} H_{\mu}(\cP^{(q)} | X),\]
where the supremum runs over countable partitions $\cP$ with $H_\mu(\cP|X)<\infty,$ and 
$\cP^{(q)}=\bigvee_{i=0}^{q-1} a^{-i}\cP$ denotes the join of the preimages $a^{-i} \cP$.
We can now state the main result of this section. The reader might benefit from reviewing Definition~\ref{def:heavy lattice} first.

\begin{prop}[Lower bound on the entropy]
\label{largeentropy}
Let $x$ be a lattice in $\cH(\eta_i)$ for some sequence $\eta_i\to 0$ and let $y\in \pi^{-1}(x)$. 
Furthermore, let $$\mu^0\in \set{\del_x^T:T\in\Z_+}^\om\cap \crly{P}(X, (\eta_i)).$$
For any $S\subset W_y^+$, where $W_y^+\subset\pi^{-1}(x)$ is the  injective unstable leaf at $y$ in the fiber,
there exists an $a$-invariant $\mu\in\crly{P}(Y)$ satisfying:
\begin{enumerate}
\item\label{pprop1} $\pi_*\mu=\mu^0$,
\item\label{pprop2} $\Supp\mu\subset\bigcap_{s\in\Z^+}\overline{\bigcup_{t\geq s}a_tS}$,
\item\label{pprop3} $h_\mu(a|X) \geq \dima S\geq\underline{\dim}_MS +h_a - \dim U^+$.
\end{enumerate}
Furthermore, if $\cP$ is any finite partition of $Y$ satisfying:
\begin{itemize}
\item $\cP$ contains an atom $P_\infty$  
of the form $\pi^{-1}(P_\infty^0)$, where $X\smallsetminus P^0_\infty$ has compact closure and 
$\psi_i|_{P_\infty^0}\equiv 0$ for some $i$ 
($\psi_i$ is as in Definition~\ref{def:heavy lattice}).
\item  $\forall P\in\cP\smallsetminus\{P_\infty\},\ \diam P<r$, with $r\in(0,\frac{1}{2})$ such that 
any $d_a$-ball of radius $3r$ has Euclidean diameter smaller than the fiber injectivity radius on $Y\smallsetminus P_\infty$,
\item $\forall P\in\cP,\ \mu(\partial P)=0$,
\end{itemize}
then, for all $q\geq 1$,
$ \frac{1}{q}H_\mu(\cP^{(q)}|X) \geq \dima S-D\eta_i,$
where $D$ is a constant depending only on the $c_i$'s and the dimension $d$.

\end{prop}

The proof of Proposition~\ref{largeentropy} will follow the strategy used to derive the variational principle for the topological entropy, as in \cite[\S 5.3.3]{ELW}, but there is a slight complication here, because the space $Y$ of grids is not compact.
To solve this problem, we will need Lemma~\ref{covering} below, which is inspired by \cite[Lemma~4.5]{elmv}.

\begin{lem}
\label{covering}
Let $P_\infty^0\subset X$ be such that $X\smallsetminus P_\infty^0$ has compact closure. Set $P_{\infty} = \pi^{-1}(P_\infty^0)$  and fix $0<r<1$ such that any $d_a$-ball of radius $3r$ has Euclidean diameter smaller than the fiber injectivity radius on $Y\smallsetminus P_\infty$.
Let $y\in Y\smallsetminus P_\infty$ and set $I=\{t\in\Z^+\ |\ a_ty\in P_\infty\}$.
For any non-negative integer $T$, let
\[ E_{y,T} = \{ z\in W_y^+ \ |\ 
\forall t\in\{1,\dots,T\}\smallsetminus I,\, d_E(a_ty,a_tz)\leq r\}.\]
Then one can cover $E_{y,T}$ by $Ce^{D|I\cap\{1,\dots,T\}|}$ $d_a$-balls of radius $\del_T = re^{-T}$,
where $C$ is a constant depending on $y$, $r$ and $a$, and $D$ is a constant depending on $a$ and the dimension $d$. 
In particular, they are independent of $T$.
\end{lem}
\begin{proof}

Before we start the proof, we make the following observation: for $y',z'\notin P_\infty$  in the same fiber,
the intersection  $B_r^{Uy',d_E}(y')\cap W_{z'}^+$ is contained in $B_r^{W_{z'}^+, d_E}(z'')$ for some $z''\in W_{z'}^+$ since the Euclidean distance $d_E$ on the fiber $Uy'=\pi^{-1}(\pi(y'))$ restricts to a Euclidean distance on the injective unstable leaf at $z'$. 
Moreover, since $r<1$ and by~\eqref{cnormalize}, 
Euclidean $r$-balls are contained in $d_a$-balls and so we conclude that
\eqlabel{ball inclusion1}{
B_r^{Uy',d_E}(y')\cap W_{z'}^+ \subset B_{r}^{W_{z'}^+, d_a}(z'').
} 
We prove the lemma by induction on $T$.\\
\underline{$T=0$:}\hspace{3pt} 
By \eqref{1110} of Remark~\ref{remark1}, the number of balls of radius $\delta_0=r$ for the metric $d_a$ needed to cover $W_y^+$ is bounded by a integer constant $C$ depending on $a$, $r$ and $y$, so that the lemma holds in this case.\\
\underline{$T-1\to T$:}\hspace{3pt}
Choose $D$ such that any $d_a$-ball of radius $\delta$ on $U^+$ can be covered by 
$e^D$ $d_a$-balls of radius $\delta/e$. Assume for clarity that $e^D$ is an integer.
By the induction hypothesis, $E_{y,T-1}$ can be covered by $N_{T-1}\defn Ce^{D|I\cap\{1,\dots,T-1\}|}$ $d_a$-balls of radius 
$\delta_{T-1}=re^{-T+1}$.

If $T\in I$, we simply cover each $\delta_{T-1}$-ball in $W_y^+$ by $e^D$ balls of radius $\delta_T = re^{-T}$ for $d_a$, and get a cover $E_{y,T}$ with cardinality $N_T=e^DN_{T-1}$ by $d_a$-balls.

If $T \notin I$, we need to cover $E_{y, T}$ by $N_T=N_{T-1}$ $d_a$-balls of radius $\delta_T$.
Denote the above cover of $E_{y,T-1}$ by $ \{ B_{\del_{T-1}}^{W_y^+,d_a}(z_i) ; i =1, \dots, N_{T-1} \}.$
As $E_{y,T} \subset E_{y,T-1}$, the set  $\{ E_{y,T} \cap B_{\del_{T-1}}^{W_y^+,d_a}(z_i) ; i =1, \dots, N_{T-1} \}$ covers $E_{y,T}$.
We claim that for each $z_i$, there exists some $p_i $ with
$$E_{y,T}\cap B_{\del_{T-1}}^{W_y^+,d_a}(z_i) \subset B_{\del_T}^{W_{z}^+,d_a}(p_i),$$
so that $E_{y,T}$ is actually covered by $\{ B_{\del_T}^{W_{z}^+,d_a}(p_i)\}$ i.e. by $N_T = N_{T-1}$ $d_a$-balls of radius $\del_T$.
Observe that
$$E_{y,T}\cap B_{\del_{T-1}}^{W_z^+,d_a}(z) \subset a_T^{-1}\pa{B_r^{Ua_Ty,d_E}(a_Ty)\cap a_TB_{\del_{T-1}}^{W_z^+,d_a}(z)}$$
By our choice of $r$, the fact that $a_Ty\notin P_\infty$ (and hence $ a_Tz\notin P_\infty$) and \eqref{eq: scaling}, the map $a_T$ scales $d_a$ by a factor of $e^T$ and we conclude that 
\eqlabel{eq:1621}{
 a_TB_{\del_{T-1}}^{W_z^+,d_a}(z) \cap B_r^{Ua_Ty,d_E}(a_Ty) = B_{er}^{W_{a_Tz}^+,d_a}(a_Tz) \cap B_r^{Ua_Ty,d_E}(a_Ty)} 
which is contained in a single $d_a$-ball of radius $r$
  by the observation~\eqref{ball inclusion1} (with $z'=a_Tz, y'=a_Ty$). Thus $E_{y,T}\cap B_{\del_{T-1}}^{W_z^+,d_a}(z) $ is contained in a single $d_a$-ball of radius $r e^{-T}$. This concludes 
the inductive step.

\end{proof}
Now we can prove Proposition~\ref{largeentropy}.

\begin{proof}[Proof of Proposition~\ref{largeentropy}]
The assumption that $\mu^0\in \set{\del_x^T:T\in\bZ_+}^\om\cap\crly{P}(X,(\eta_i))$ means that we may fix an increasing sequence of integers $(n_k)$ such that 
$$\mu^0_k\defn\frac{1}{n_k}\sum_{n=0}^{n_k-1}\delta_{a^nx}\wstar \mu^0\in \crly{P}(X,(\eta_i)).$$
Then, for each $k\geq 1$, let $S_k$ be a maximal $\rho_k$-separated subset of $S$, for the metric $d_a$, where $\rho_k\defn e^{-n_k}$.
Let $\nu_k\defn\frac{1}{|S_k|}\sum_{y\in S_k}\delta_y$ be the normalized counting measure on $S_k$ and
\[ \mu_k \defn \frac{1}{n_k}\sum_{n=0}^{n_k-1} a^n_*\nu_k.\]
Since $\pi:Y\to X$ is proper, and the measures $\pi_*\mu_k=\mu^0_k$ converge to a probability measure $\mu^0$ on $X$,
we conclude that the sequence of measures $(\mu_k)$ is tight, i.e. that
any weak-* converging subsequence of it converges to a probability measure.
Extracting a subsequence if necessary, we may assume without loss of generality that $(\mu_k)$ converges weak-* to some 
probability measure, which we denote by $\mu$. By continuity of $\pi_*$ we obtain 
$$\pi_*\mu =\pi_*\lim_k \mu_k = \lim_k \mu_k^0 = \mu^0$$
which is the item \eqref{pprop1} in the proposition.

By construction, $\Supp\mu$ is contained in the set of accumulation points of the forward orbit of $S$ under $(a_t)$ which establishes item \eqref{pprop2} in the proposition.

The right inequality in item \eqref{pprop3} follows directly from Lemma~\ref{relating dimensions}. 
For simplicity of notation, let $\beta\defn\dima S$, so that
\begin{equation}\label{eq:1532}
 \liminf_{k\to\infty}\frac{\log|S_k|}{n_k} \geq \beta.
 \end{equation}
To prove that $\mu$ also satisfies the the left inequality in item \eqref{pprop3} of the proposition saying that 
$h_\mu(a|X)\ge \beta$ we proceed as follows. 
Given $i$ 
we construct a partition $\cP$ of $Y$ for which 
\begin{equation}\label{eq:2326}
\frac{1}{q}H_\mu(\cP^{(q)}|X) \geq \beta -D\eta_i,
\end{equation}
where $D$ is as in Lemma~\ref{covering}. This implies that 
$$h_\mu(a|X)\ge h_\mu(a,\cP|X) =\lim_q \frac{1}{q}H_\mu(\cP^{(q)}|X)\ge \beta -D\eta_i,$$
and letting $i\to\infty$ we obtain the desired inequality $h_\mu(a|X)\ge \beta.$

To this end fix $i_0$ and consider $\eta=\eta_{i_0}$.  
Choose a set $P^0_\infty \subset X$ such that $X\smallsetminus P_\infty^0$ is compact, $\mu^0(\partial P_\infty^0) = 0$, 
and such that $\psi_{i_0}|_{P_0^\infty}\equiv 0$. Since $\mu^0\in \crly{P}(X,(\eta_i))$ it follows that $\mu^0(P_\infty^0)\le \eta_{i_0}=\eta$ and in turn,
$P_\infty = \pi^{-1} (P^0_\infty)$ satisfies $\mu(P_\infty)\leq\eta$ and $\mu(\partial P_\infty)=0$.
Since $(\mu_k)$ converges to $\mu$, we have for any $y\in\pi^{-1}(x)$, with $I=\{t\in\Z^+\ |\ a_ty\in P_\infty\}$ (note that $I$ depends only on $x$), 
\begin{equation}\label{eq:1530}
 \limsup_{T\to\infty} \frac{1}{n_k}|I\cap\{1,\dots,n_k\}| \leq \eta.
 \end{equation}
Then, let $r\in(0,\frac{1}{2})$ be as in Lemma~\ref{covering} and complement $P_\infty$ to a finite partition $\cP=\{P_\infty,P_1,\dots,P_\ell\}$ of $Y$ such that for every atom $P_i\neq P_\infty$ and every $x\in X$, the Euclidean diameter of $P_i\cap\pi^{-1}(x)$ is at most $r$, and such that for each $P\in\cP$, $\mu(\partial P)=0$, where $\partial P$ denotes the boundary of $P$.
To build such $\cP$, observe that around each point $z$ in $Y\setminus P_\infty$, there exists a ball $B_z$ around $z$ in $Y$ such that $\mu(\partial B_z)=0$, and 
\begin{equation}\label{diamfib}
\forall x\in X,\  \diam_E(B_z\cap\pi^{-1}(x))<r.
\end{equation}
A finite cover of $Y\smallsetminus P_\infty$ by such balls generates the desired partition by a simple disjointification procedure.

For $q\geq 1$, let $\cP^{(q)}=\bigvee_{p=0}^{q-1}a^{-p}\cP$.
For $n_k$ large, write the Euclidean division of $n_k-1$ by $q$
\[ n_k-1=qn'+s,\ \mbox{with}\ s\in\{0,\dots,q-1\}.\]
By subadditivity of the entropy with respect to the partition, for each $p\in\{0,\dots,q-1\}$,
\begin{align*}
H_{\nu_k}(\cP^{(n_k)}|X)
& \leq H_{a^{p}\nu_k}(\cP^{(q)}|X)+H_{a^{p+q}\nu_k}(\cP^{(q)}|X)+\dots\\
& \qquad\qquad +H_{a^{p+qn'}\nu_k}(\cP^{(q)}|X)+2q\log|\cP|.
\end{align*}
Summing those inequalities for $p=0,\dots,q-1$, and using the fact that entropy is a concave function of the measure, we obtain
\begin{align*}
qH_{\nu_k}(\cP^{(n_k)}|X)
& \leq \sum_{n=0}^{n_k-1}H_{a^n\nu_k}(\cP^{(q)}|X) + 2q^2\log|\cP|\\
& \leq n_k H_{\mu_k}(\cP^{(q)}|X) + 2q^2\log|\cP|
\end{align*}
and therefore
\begin{equation}\label{almost}
\frac{1}{q}H_{\mu_k}(\cP^{(q)}|X)
\geq \frac{1}{n_k}H_{\nu_k}(\cP^{(n_k)}|X) - \frac{2q\log|\cP|}{n_k}.
\end{equation}
Now since $\nu_k$ is supported on a single atom of the $\sigma$-algebra $\pi^{-1}(\cB_X)$, we have $H_{\nu_k}(\cP^{(n_k)}|X)=H_{\nu_k}(\cP^{(n_k)})$.
Moreover, we claim that
\begin{equation}\label{atoms}
H_{\nu_k}(\cP^{(n_k)}) \geq \log|S_k|-D|I\cap\{1,\dots,n_k\}| -D-\log C,
\end{equation}
where $C,D$ are the constants given by Lemma~\ref{covering}.
To see this, it suffices to show that an atom of $\cP^{(n_k)}$ contains at most $Ce^{D(|I\cap \{1,\dots,n_k\}|+1)}$ points of $S_k=\Supp\nu_k$.
This follows from Lemma~\ref{covering}.
Indeed, Equation (\ref{diamfib}) implies that
if $P$ is any non-empty atom of $\cP^{(n_k)}$, fixing any $y\in P$,
\[ S_k \cap P= S_k \cap [y]_{\cP^{(n_k)}} \subset E_{y,n_k-1}\] can be covered by $Ce^{D(|I\cap\{1,\dots,n_k\}|+1)}$ many $re^{-n_k}$-balls for $d_a$. 
Since $S_k$ is $\rho_k=e^{-n_k}$-separated with respect to $d_a$ and $r<\frac{1}{2}$, we get $$\card(S_k\cap[y]_{\cP^{(n_k)}})\leq Ce^{D(|I\cap\{1,\dots,n_k\}|+1)}.$$

Going back to (\ref{almost}), we find
\[ \frac{1}{q}H_{\mu_k}(\cP^{(q)}|X)
\geq \frac{1}{n_k}\big(\log|S_k| - D|I\cap\{1,\dots,n_k\}|
- D -\log C - 2q^2\log|\cP|\big).\]
Now the atoms of $\cP$ -- and hence of $\cP^{(q)}$ -- satisfy $\mu(\partial P)=0$, so we may let $k$ go to infinity
to obtain equation~\eqref{eq:2326} using \eqref{eq:1532} and \eqref{eq:1530}.

Finally the second part of the proposition regarding partitions satisfying the bullet-requirements follows by reviewing the proof
of~\eqref{eq:2326} and noting that the only properties of the constructed partition $\cP$ we used are those in the bullet list.
\end{proof}

\section{Maximal entropy and invariance}
\label{sec:invariance}
\label{sec:invariance}
In this section we recall some concepts and results from \cite{EL} and explain how they imply 
the following proposition, which is essential for the proof of Theorem~\ref{thm:main}.

\begin{prop}[Maximal entropy implies $U$-invariance]
\label{maxinv}
Let $\mu$ be an $a$-invariant probability measure on $Y$.
Then 
\[ h_\mu(a|X) \leq h_a\]
with equality if and only if $\mu$ is $U$-invariant.
\end{prop}
To prove Proposition~\ref{maxinv} we relate the \enquote{dynamical} relative entropy $h_\mu(a|X)$  
to some  relative 
\enquote{static} entropy $H_\mu(\cA_1|\cA_2)$
where the $\cA_i$ are countably generated $\sig$-algebras that encode the dynamics. 
For definitions and elementary properties of relative entropies of $\sig$-algebras we refer the reader to~\cite[Chapter 2]{ELW}.

\begin{defi}[7.25. of \cite{EL}]Let 
$G^- \defn \{ g\in G\ |\ a_tga_t^{-1}\to e\ \mbox{as}\ t\to\infty\}$ be the stable horospherical subgroup associated to $a$ and let $U^- = U\cap G^-$.
Let $\mu$ be an $a$-invariant measure on $Y$ and $U<G^-$ a closed $a$-normalized subgroup.
\begin{enumerate}
\item We say that a countably generated $\s$-algebra $\cA$ is \emph{subordinate to $U$} (mod $\mu$)  if for $\mu$-a.e. $y$, there exists $\de > 0$ such that
\[ B^{U}_\de\cdot y \subset [y]_{\cA} \subset B^{U}_{\de^{-1}}\cdot y.\] 
\item
We say that $\cA$ is \emph{$a$-descending} if $a^{-1}\cA \subset \cA$.
\end{enumerate}
\end{defi}
\begin{thm}[Einsiedler-Lindenstrauss]\label{thm:EinLin}
Let $\mu$ be an $a$-invariant probability measure on $Y$.
If $\cA$ is a countably generated sub-$\sig$-algebra of the Borel $\sig$-algebra which is $a$-descending and $U^-$-subordinate
then $H_\mu(\cA|a^{-1}\cA)\le h_a$ with equality if and only if $\mu$ is $U^-$-invariant.
\end{thm}
\begin{proof}
By considering the ergodic decomposition one sees that it is enough to prove the theorem in the case
$\mu$ is ergodic. Under the ergodicity assumption the statement follows directly from combining
\cite[Proposition~7.34]{EL} and  \cite[Theorem~7.9]{EL}.
\end{proof}

The following lemma furnishes the link between 
the relative dynamical entropy $h_\mu(a|X)$ and the relative entropy $H_\mu(\cA_1|\cA_2)$ for suitable 
$\cA_i$.
If $\cP$ is a partition of $Y$, we write for any integer $m$,
\[
\cP_m^\infty=\bigvee_{k=m}^\infty a^{-k}\cP.
\]
Recall also that a partition $\cP$ is said to be a two-sided generator for $a$ with respect to an $a$-invariant probability measure 
if the Borel $\sigma$-algebra is generated up to null sets by the union of all partitions $\bigvee_{-m}^m a^{-k}\cP$, $m\geq 1$.

\begin{lem}
\label{reco}
Assume that $\mu$ is an $a$-invariant probability measure on $Y$ and $\cP$ is a countable partition that is a two-sided generator for $a$ with respect to  $\mu$.
Let $\cA$ be the $\sigma$-algebra generated by $\cP_0^\infty\vee\pi^{-1}(\cB_X)$.
Then
\[ h_\mu(a|X) = H_\mu(\cA|a^{-1}\cA).\]
\end{lem}
\begin{proof}
Let $\cP$ be a countable partition that is a two-sided generator for $a$.
By \cite[Proposition~2.19 and Theorem~2.20]{ELW}, writing $\cC$ for $\pi^{-1}(\cB_X)$, we have
\[ h_\mu(a|X) = H_\mu(\cP|\cP_1^\infty\vee\cC)
= H_\mu(\cP_0^\infty\vee\cC|\cP_1^\infty\vee\cC).\]
Since $\cC$ is $a$-invariant, we indeed find
\[ h_\mu(a|X) = H_\mu(\cA|a^{-1}\cA).\]
\end{proof}
We now prove Proposition~\ref{maxinv} by constructing (almost by citation) a two-sided generator $\cP$ for $a$ modulo 
$\mu$ with the property that $\cP_0^\infty\vee \pi^{-1}(\cB_X)$ is $U^-$-subordinate, and then the statement follows
by combining Theorem~\ref{thm:EinLin} with Lemma~\ref{reco}.

\begin{proof}[Proof of Proposition~\ref{maxinv}]
Writing the ergodic decomposition $\mu=\int\mu_y^{\cE}\,\mathrm{d}\mu(y)$, we have
\[ h_\mu(a|X)=\int h_{\mu_y^{\cE}}(a|X)\,\mathrm{d}\mu(y),\] 
so it is enough to prove the proposition for $\mu$ ergodic.

We can then use \cite[Proposition~7.44]{EL} to obtain a countable partition $\cP$ that is a generator for $a\mod\mu$, and such that 
$\cP_0^\infty$
is $a$-descending and subordinate to $G^-$.
Let
\[ \cC \defn \pi^{-1}(\cB_X) \quad\mbox{and}\quad
 \cA \defn \cC\vee\cP_0^\infty.\]
Then $\cA$ is clearly countably generated and $a$-descending, and we claim that it is also $U^-$-subordinate.
Indeed, since $[y]_{\cC}=\pi^{-1}(\pi(\{y\}))$ is equal to the orbit of $y$ under the unipotent radical $U$, we get $[y]_{\cA} = U\cdot y \cap [y]_{\cP_0^\infty}$.
But $\cP_0^\infty$ is subordinate to $G^-$, and for $\delta>0$,
\eqlabel{inclusion2}{
B_{\delta_1}^{U^-}\cdot y \subset B_\delta^{G^-}\cdot y\cap U\cdot y
\quad\mbox{and}\quad
B_{\delta^{-1}}^{G^-}\cdot y\cap U\cdot y \subset B_{\delta_1^{-1}}^{U^-}\cdot y,
}
for some constant $\delta_1$ depending on $\delta$ and $y$, because $G^-y\cap Uy=U^-y$.
This shows that, for almost every $y$, for some $\delta>0$,
\[ B_\delta^{U^-}\cdot y \subset [y]_{\cA}
\subset B_{\delta^{-1}}^{U^-}\cdot y,\]
i.e. that $\cA$ is subordinate to $U^-$.

By Lemma~\ref{reco} we have $h_\mu(a|X)=H_\mu(\cA|a^{-1}\cA)$, and so Theorem~\ref{thm:EinLin} shows that 
$h_\mu(a|X)\le h_a$ with equality if and only if $\mu$ is $U^-$-invariant. Moreover, if  
$h_\mu(a|X)= h_a$ then 
$$h_\mu(a^{-1}|X) =h_\mu(a|X)= h_a = h_{a^{-1}},$$
where the last equality follows from the fact that $\sum_{j=1}^d c_j = 0$. 
We then apply the same reasoning to $a^{-1}$ 
and find that $\mu$ is also $U^+$-invariant. 
This completes the proof since $U$ is generated by $U^+$ and $U^-$, because we assume that for any $1\le j\le d$, $c_j\ne 0$.

%
\end{proof}

\section{Proof of the main theorem}
\label{sec:mt}

Using the results of the preceding two sections, we now state and prove 
Theorem~\ref{mainthm}, which is the main result of this article.
We deduce Theorem \ref{thm:main} as a corollary.

\subsection{$\cO$-avoiding grids}
Before we can state the more precise version of Theorem~\ref{thm:main} that we will derive here, we need to set up some notation.

For a bounded open set $\cO\subset \bR^d$ and a heavy lattice $x$, our goal is to bound the dimension of 
the set 
$$\Ybe(x) = \set{y\in \pi^{-1}(x) : y \cap (\cup_{t\ge 0} a_{-t}\cO) \textrm{ is finite}}.$$ 
For an interval $I\subset \bR$, let  
\[ E^\cO_I = \set{y \in Y: y\cap (\cup_{t\in I} a_{-t}\cO)=\varnothing} = \set{y\in Y: \forall t \in I, a_ty\cap \cO = \varnothing}\]
and
\[E^\cO_I(x) =E^\cO_I\cap \pi^{-1}(x).\] 
It is more natural from the dynamical point of view to work with $E^\cO_\bR$ than with $\Ybe$ because it is $a_t$-invariant and closed. However, we insist on working with finite intersections instead of empty intersection to obtain $\liminf$ statements in our applications rather than $\inf$. Note that for a vector $v\in\bR^d$, either $a_tv\to\infty$ or $v_i=0$ for all $i\in J^+$ in which case 
$a_tv\to 0$. Since $\cO$ is bounded, we deduce that  
\eqlabel{eq:inclusion3}{
\Ybe(x) \subset \bigcup_{r\geq 0}\Ybr(x) \cup \set{y\in\pi^{-1}(x): \exists v\in y, \forall i\in J^+, v_i=0}.
}
Since  
\eqlabel{eq:lower dimensional}{
\dim_H \set{y\in\pi^{-1}(x): \exists v\in y, \forall i\in J^+, v_i=0} = d-\av{J^+}\le d-1.
}
we will focus on bounding the dimension of each $\Ybr(x)$. Finally, a nice feature of the set $\Ybr$ is the following simple 
observation which can be verified by the reader.
\begin{lem}\label{lem:gaining invariance} 
Let $y\in \Ybr$ and suppose $z$ is an accumulation point of the forward orbit $\set{a_ty}_{t>0}$. Then $z\in E^\cO_\bR$. 
In particular, any measure $\mu$ obtained by averaging along the forward trajectory of $y$ is supported in $E^\cO_\bR$.
\end{lem}
We are now in a position to state and prove our main results.
\begin{thm}[Heavy lattices have few bad grids]
\label{mainthm}
Fix a sequence $\eta_i\to 0$, and a bounded open set $\cO$ in $\bR^d$. 
Then, there exists $\del>0$ such that for any $r>0$ and any  $x\in \cH(\eta_i)$,
\[\dimm \Ybr(x)\leq d-\delta.\]
\end{thm}
\begin{proof}[Proof of Theorem~\ref{thm:main}]
Using \eqref{eq:inclusion3} and \eqref{eq:lower dimensional}, the theorem follows at once from Theorem~\ref{mainthm}, because $\E{r}$ is increasing in $r$.
Note that as opposed to $\dimm$, the Hausdorff dimension of a countable union of sets is bounded above by any bound on the individual 
dimensions, hence our passage to Hausdorff dimension (see also Remark~\ref{rem:dim} below).
\end{proof}

\begin{rem}\label{rem:dim}
The set $\Ybe(x)$ is dense is $\pi^{-1}(x)$, so that $\dimm\Ybe(x)=d$. But its Hausdorff dimension is strictly less than $d$.
\end{rem}

\begin{proof}[Proof of Theorem~\ref{mainthm}]
Fix $\eps>0$.
We argue by contradiction and assume the following:
\eqlabel{eq:dimbound1}{
\forall m\geq 1,\,\exists r_m>0, \exists x_m\in\cH(\eta_i):\  \dimm\Ybm(x_m) \ge d-\frac{1}{m}.
}
Fix a smaller open set $\cO'$ whose closure is in $\cO$. Set $u = \dim U^+$.

\noindent\tb{Claim 1.}  \textit{ For any large enough $m$, there exists a grid $y_m\in \pi^{-1}(x_m)$ such that 
the injective unstable leaf $W_{y_m}^+$ satisfies $\dimm ( W_{y_m}^+\cap \Ybmprime)\ge u - \frac{1}{m}.$}
Recall that $U^-= U \cap G^- = U \cap \{ g\in G\ |\ a_tga_t^{-1}\to e\ \mbox{as}\ t\to\infty\}$.
For any grid $y\in Y$, any $v\in U^-$ of norm $\le1$ and  $t>0$, 
the two grids $a_{t}y$ and $a_{t}vy$ differ by a translation in the direction of $U^-$ which is of norm $\le e^{-\al t}$, 
where $\al = \min\set{\av{c_i}: i\in J^- }>0$. We deduce that if $a_t y\cap \cO=\varnothing$ and $t$ is large enough, then $a_tvy\cap \cO'=\varnothing$. 
In particular,
for all $m$ large enough and all $v\in U^-$ of norm $\le 1$, we have $v \Ybm \subset \Ybmprime$. 
Therefore,
$$\dimm ( W_{y_m}^+\cap \Ybmprime) + \dim U^- \ge d - \frac{1}{m}.$$
Since $c_i\ne 0$ for all $i$ 
, $d-\dim U^- = u$ and Claim 1 follows.

Let $S_m = W_{y_m}^+\cap \Ybmprime$ be as in Claim 1. The first part of Proposition~\ref{largeentropy} yields an $a$-invariant probability measure $\mu_m$ with
\begin{enumerate}
\item\label{compact proj} $\pi_*\mu_m\in \crly{P}(X,(\eta_i))$.
\item\label{nice to have} $\Supp\mu_m \subset E^{\cO'}_\bR$ (by Proposition~\ref{largeentropy} and Lemma~\ref{lem:gaining invariance})
\item $h_{\mu_m}(a|X) \geq h_a-\frac{1}{m} $.
\end{enumerate}
The fact that $\pi:Y\to X$ is proper together with Lemma~\ref{rem:1511}\eqref{r02} which says that $\crly{P}(X,(\eta_i))$ is compact we deduce 
that after taking a subsequence of $\mu_m$ we may assume it converges  
to some $a$-invariant probability measure $\mu$ such that $\pi_*\mu\in \crly{P}(X,(\eta_i))$. 
Since $\Supp\mu_m \subset E^{\cO'}_\bR$ and $E^{\cO'}_\bR$ is closed, 
we conclude that $\mu$ 
is supported in $E^{\cO'}_\bR$. 
Our next goal is to show the following.

\noindent\tb{Claim 2}. \textit{The measure $\mu$ is $U$-invariant}.

To prove Claim 2 we apply the second part of Proposition~\ref{largeentropy} to the measures $\mu_m$ simultaneously in the following manner.
Fix $i_0$ and let $\eta = \eta_{i_0}$. Let $\cP$ be a finite partition of $Y$ 
satisfying:
\begin{itemize}
\item $\cP$ contains a single unbounded atom $P_\infty$ and it is 
of the form $\pi^{-1}(P_\infty^0)$, where $\psi_{i_0}|_{P_\infty^0}\equiv 0$ ($\psi_i$ is as in Definition~\ref{def:heavy lattice}).
\item  $\forall P\in\cP\smallsetminus\{P_\infty\},\ \diam P<r$, with $r\in(0,\frac{1}{2})$ such that 
any $d_a$-ball of radius $3r$ has Euclidean diameter smaller than the fiber injectivity radius on $Y\smallsetminus P_\infty$,
\item $\forall P\in\cP,\ \nu(\partial P)=0$ for $\nu\in\set{\mu,\mu_m:m\in\bN}$.
\end{itemize}
A similar partition was built in the proof of Proposition~\ref{largeentropy}. The only difference here is that in the third bullet
here we demand that the boundaries of the atoms of $\cP$ will be simultaneous null sets for more than one measure. Since
we are only requiring this for a countable collection of measures this is easily achieved.

From the second part of Proposition~\ref{largeentropy} we deduce that 
for any $m$, for all $q\geq 1$,
\begin{equation}\label{finite}
\frac{1}{q}H_{\mu_m}(\cP^{(q)}|X) \geq h_a-\frac{1}{m}-D\eta_{i_0}.
\end{equation}
Since the boundary of the atoms of $\cP^{(q)}$ are $\mu$-null, we can pass to the limit as $m\to\infty$ in ~\eqref{finite}
and deduce that for any $q$, $\frac{1}{q}H_{\mu}(\cP^{(q)}|X) \geq h_a-D\eta_{i_0}.$ Taking $q\to\infty$ and $i_0\to\infty$ we deduce that 
$$h_\mu(a|X)\ge \lim \frac{1}{q}H_\mu(\cP^{(q)}|X) \ge  h_a.$$ 
By Proposition~\ref{maxinv} we deduce that equality holds and that $\mu$ is $U$-invariant, as claimed.

We arrive at the desired contradiction because $\mu$ is supported in $E^{\cO'}_\bR$,
which cannot contain a full fiber: 
given a lattice $x$, the grids of $x$ which contain points in $\cO'$ cannot be in $E^{\cO'}_\bR$.
\end{proof}

\section{Diophantine approximation}
\label{sec:diophantine}

In this section we prove Theorem~\ref{thm:app}, which, in fact, will follow from a sharper result in the spirit of Theorem~\ref{mainthm}.
We also reformulate and generalize the result in terms of approximation of affine subspaces of $\R^n$ by integer points.

\subsection{Inhomogeneous Diophantine approximation of vectors in $\R^n$}
Fix a dimension $n\ge 1$ and let $d \defn n+1$. For clarity of exposition, we start 
with the diagonal flow $a_t=\diag{e^t,\dots,e^t,e^{-nt}}$. 
Recall that given a vector $v\in\R^n$, we let
\[ \mb{Bad}^\eps(v) = \{ w\in\R^n\ :\ 
\liminf_{k\to\infty}k^{1/n}\langle kv-w\rangle \geq \eps\}.\]
Given a vector $v\in \bR^n$ we let $x_v\defn \smallmat{I_n&v\\0&1}x_0\in X$, where $x_0$ denotes the identity coset, which represents the standard lattice $\bZ^d$. 
The diophantine properties of the vector $v$ are usually captured by the dynamics of the lattice $x_v$. For example, singularity
of $v$ is equivalent to the divergence of the orbit $(a_tx_v)_{t>0}$. In analogy with Definition~\ref{def:heavy lattice} we make the following.

\begin{defi}\label{def:heavy vector}
A vector $v\in \bR^n$ is said to be \emph{heavy} if the lattice $x_v$ is heavy according to Definition~\ref{def:heavy lattice}.
\end{defi}
A nice exercise
is the following characterization of heaviness  of a number $\al\in\bR$ in terms of the continued fraction expansion of $\al$.
\begin{xca}
Show that a number $\al=[a_0;a_1,a_2,\dots]$ is
heavy if and only if
$$\forall \del>0\; \exists \eps>0 \textrm{ such that }\liminf_{N\to\infty}\frac{1}{N}\sum_{k=1}^N\max\set{\log \eps a_k, 0} \le \del.$$
\end{xca}

We prove the following result, which will easily imply Theorem~\ref{thm:app}.
\begin{thm}[Heavy vectors have few badly approximable points]
\label{thm:appug}
If $v\in\bR^n$ is heavy then for any $\eps>0$, $\dim_H(\mb{Bad}^\eps(v))< n$.
In fact, if $\eta_i\to0$ is a sequence of non-negative numbers, then for any $\eps>0$ there exists $\del = \del(\eps, (\eta_i))>0$ such that for any $v\in \bR^n$ for 
which $x_v\in \cH(\eta_i)$,
$\dim_H(\mb{Bad}^\eps(v))\le n-\del$.
\end{thm}
\begin{proof}
We write vectors in $\bR^d=\bR^n\times \bR$ as $w_s\defn \smallmat{w\\ s}$ with $w\in\bR^n, s\in\bR$.
Let $(\eta_i)$ be as in the statement and let $v\in\bR^n$ be such that $x_v\in\cH(\eta_i)$. 
Let $\eps>0$ and let $\cO \defn \set{w_s\in\bR^d: s\in(0,1), \norm{w}<\frac{\eps}{2}}$. 
Note that 
\begin{equation}\label{eq:spike2}
S^+(\cO) = \set{ w_s: s\ge 1, s^{1/n}\norm{w}<\frac{\eps}{2}} \cup \cO.
\end{equation}
We know by Theorem~\ref{thm:main} that there exists $\del = \del(\eps,(\eta_i))>0$ such that $\dim_H F_{S^+(\cO)}(x_v)\le d-\del$.
We will show that for any $w\in\mb{Bad}^\eps(v)$ and for any $s\in[0,1]$, the grid $x_v-w_s$ belongs to $F_{S^+(\cO)}(x_v)$.
This will finish the proof.

To this end, let $w\in\mb{Bad}^\eps(v)$ and $s\in[0,1]$. Note that 
$$x_v-w_s 
 =  \underset{k \in \Z} {\bigcup} \set{\mat{\vec{m}+kv-w\\ k-s}: \vec{m}\in\bZ^n}.$$
We call $k$ the \emph{layer parameter} of $\set{\mat{\vec{m}+kv-w\\ k-s}: \vec{m}\in\bZ^n}.$  Note that 
the set of vectors in each layer is discrete. 
Therefore if we suppose that $x_v-w_s$ intersects $S^+(\cO)$ in infinitely many points, then we conclude that $S^+(\cO)$ must contain points
in arbitrarily high layers (i.e.\ with $k$ arbitrarily large). In particular, the description
of $S_+(\cO)$ given in~\eqref{eq:spike2} implies that there exist arbitrarily large $k>0$ and vectors $\vec{m}\in \bZ^n$
such that $(k-s)^{1/n}\norm{\vec{m}+kv -w} \le \frac{\eps}{2}$. In particular, $\liminf_{k\to\infty} k^{1/n}\idist{kv-w}<\eps$ and 
so $w\notin \mb{Bad}^\eps(v)$ contradicting our assumption. We deduce that $\#\set{x_v-w_s \cap S^+(\cO)} < \infty$, i.e. $x_v-w_s\in F_{S^+(\cO)}(x_v)$ as claimed.
\end{proof}

As a corollary, we now derive Theorem~\ref{thm:app} from the introduction, which we recall here for convenience.

\begin{cor}
For any $\eps>0$, there exists $\del>0$ such that for almost every $v\in \R^n$,
$\dim_H\mb{Bad}^\eps(v) < n-\del$.
\end{cor}
\begin{proof}
The proof is similar to the proof of Corollary~\ref{cor:main}. It is well known that 
$$\Om = \set{v\in\bR^n : \del_{x_v}^T\wstar m_X}$$ 
has full Lebesgue measure.  Using Lemma~\ref{rem:1511}\eqref{r01} we choose 
$(\eta_i)$ so that $m_X\in\crly{P}(X,(\eta_i))$. We then have by definition that 
$\set{x_v: v\in \Om}\subset \cH(\eta_i)$. The theorem thus follows from Theorem~\ref{thm:appug}.
\end{proof}

Theorem~\ref{thm:appug} can be generalized in several ways, using Theorem~\ref{mainthm} for more general flows $(a_t)$.
For example, if $(i_1,\dots,i_n)$ is an $n$-tuple of real numbers such that
\[ \forall \ell,\, i_\ell\in(0,1)
\quad\mbox{and}\quad
\sum_{\ell=1}^ni_\ell=1,\]
we can define, for any vector $v=
\left(\begin{matrix}v_1\\ \vdots\\ v_n\end{matrix}\right)\in\R^n$,
\[\Bad^\eps_{(i_1,\dots,i_n)}(v)
\defn \{ w=\,\left(\begin{matrix}w_1\\ \vdots\\ w_n\end{matrix}\right)\in\R^n\ :\ \forall \ell,\, \liminf_{k\to\infty} k^{i_\ell}\langle kv_\ell-w_\ell\rangle\geq\eps\},\]
and
\[\Bad_{(i_1,\dots,i_n)}(v)=\bigcup_{\eps>0}\Bad^\eps_{(i_1,\dots,i_n)}(v).\]
It is known (see for example~\cite{PV, KTV, KWAdvances}) that for any $v\in\R^n$, $\dim_H\Bad_{(i_1,\dots,i_n)}(v)=n$. 
Theorem~\ref{mainthm} applied with the flow $a_t=\diag{e^{i_1t},\dots,e^{i_nt},e^{-t}}$ yields the following.

\begin{thm}[Heavy vectors for weighted approximation]
\label{thm:app2}
Let $v\in\R^n$ be heavy for $a_t=\diag{e^{i_1t},\dots,e^{i_nt},e^{-t}}$.
Then, for all $\eps>0$,
\[ \dim_H\Bad^\eps_{(i_1,\dots,i_n)}(v)<n.\]
Moreover, for any $\eps>0$, there exists $\del>0$ such that for a.e. $v\in\R^n$,
\[ \dim_H\Bad^\eps_{(i_1,\dots,i_n)}(v)<n-\del.\]
\end{thm}

Being very similar to that of Theorem~\ref{thm:appug}, the proof of Theorem~\ref{thm:app2} is left to the reader.
\subsection{Approximation of affine subspaces}

Another natural generalization of Theorem~\ref{thm:appug} is obtained by replacing the vector $v$ by a matrix. We choose to present this generalization in a \emph{projective} manner which is not common but we find it very natural. That is, in the context of Diophantine approximation of affine subspaces of $\R^d$ by points in $\Z^d$. The case of Theorem~\ref{thm:appug} corresponding to the subspace being a line (see Remark~\ref{rem:1807}).

Let $\Grass(\ell,d)$ be the Grassmannian of $\ell$-dimensional linear subspaces of $\R^d$.
Recall that by Minkowski's first theorem on convex bodies, for every $W_0\in\Grass(\ell,d)$, the inequality
\[ d(\bbk,W_0) \leq 2^d\cdot \|\bbk \|^{\frac{-\ell}{d-\ell}}\]
has infinitely many solutions $\bbk \in\Z^d$, where $\|\bbk \|$ denotes the Euclidean norm of $\bbk$.
It is therefore natural to say that an affine subspace $W$ of dimension $\ell$ in $\R^d$ is \emph{$\eps$-badly approximable} if it satisfies
\[ \liminf_{\substack{\bbk \to\infty\\\bbk \in\Z^d}} \|\bbk \|^{\frac{\ell}{d-\ell}}d(\bbk,W)\geq\eps.\]

Let $\Grass_A(\ell,d)$ denote the Grassmannian of $\ell$-dimensional affine subspaces of $\R^d$ and $\pi:\Grass_A(\ell,d)\to\Grass(\ell,d)$ the natural projection, mapping an affine subspace to its linear part.
For a linear subspace  $W_0\in\Grass(\ell,d)$ of $\R^d$, we want to study the set
\[ \Badg^\eps(W_0) \defn \{ W\in \pi^{-1}(W_0)\ |\ \liminf_{\substack{\bbk \to\infty\\\bbk\in\Z^d}} \|\bbk\|^{\frac{\ell}{d-\ell}}d(\bbk,W)\geq\eps\}\]
of $\eps$-badly approximable affine subspaces $W\leq\R^d$ with linear part $W_0$.
It is known that $\dim_H (\Badg(W_0))=d-\ell$, where $\Badg(W_0))\defn \bigcup_{\eps>0}\Badg^\eps(W_0)$. See \cite{ET}.

\begin{rem}\label{rem:1807}
Let $n=d-1$.
For $v\in\R^n$, consider the line $W_0\in\R^d$ spanned by the vector $\tilde{v}=\left(\begin{matrix}1\\ v\end{matrix}\right)\in\R^d$.
Then a vector $w\in\R^n$ is in $\Bad(v)$ if and only if $\tilde{w}+W_0$ is a badly approximable line in $\R^d$, so that the setting of the previous subsection corresponds to Diophantine approximation of lines in $\R^d$.
\end{rem}

\begin{thm}[Approximation of affine subspaces]
\label{thm:app3}
For all $\eps>0$, there exists $\delta>0$ such that for almost every $W_0\in\Grass(\ell,d)$,
\[ \dim_H\Badg^\eps(W_0) \leq d-\ell-\del.\]
\end{thm}
\begin{proof}
Since the proof is very similar to that of Theorem~\ref{thm:appug}, we keep it terse.
We apply Theorem~\ref{mainthm} with flow
\[ a_t=\diag{e^t,\dots,e^t,e^{-\frac{\ell t}{d-\ell}},\dots,e^{-\frac{\ell t}{d-\ell}}}.\]
Let $W_0\in\Grass(\ell,d)$, and choose $g_{W_0}\in G_0=\SL_d(\R)$ such that $g_{W_0}\cdot W_0=\Span(e_1,\dots,e_\ell)$.
For almost every $W_0$, the orbit $(a_tg_{W_0}\Z^d)_{t>0}$ equidistributes in $X$ (note that this property does not depend on our choice of $g_{W_0}$).
Taking $\cO=B_{\frac{\eps}{2}}$ to be the open ball of radius $\eps/2$ in $\R^d$, Theorem~\ref{mainthm} shows that there exists $\delta>0$ such that for almost every $W_0$,
\[ \dim_H F_{S^+(B_{\frac{\eps}{2}})}(g_{W_0}\Z^d) \leq d-\delta.\]
Assume now that $W$ is an affine subspace with linear part $W_0$, and choose $g_W\in G=\ASL_d(\R)$ such that $g_W\cdot W=\Span(e_1,\dots,e_{\ell})$ (as an affine subspace).
It is a simple computation to check that if $W\in\Badg^{\eps}(W_0)$, then the grid $g_W\Z^d$ lies in $F_{S^+(B_{\frac{\eps}{2}})}(g_{W_0}\Z^d)$, independently of our choice of $g_W$.
The above bound on the Hausdorff dimension of $F_{S^+(B_{\frac{\eps}{2}})}(g_{W_0}\Z^d)$ therefore implies
\[ \dim_H\Badg^\eps(W_0) \leq d-\ell-\del.\]
\end{proof}

\section{Examples}
\label{sec:ex}

In this section, to justify the necessity of some non-escape-of-mass assumption on $x$ in Theorem~\ref{mainthm}, we construct non-singular lattices in $\R^2$ with lots of bad grids: $x$ with non-divergent orbits but for which there exists $\eps>0$ such that $\Ybe(x)$ has full Hausdorff dimension in $\pi^{-1}(x)$ for a suitable choice of $\cO$.

\begin{prop}[Lattices with lots of bad grids]
\label{example}
There exists a non-singular unimodular lattice $x$ and an open bounded set $\cO$ such that
\[ \dim_H \Ybe(x) = 2.\] 
\end{prop}

Fix a lattice $x$ in $\R^2$, assume $\lambda_1(x)\geq\frac{1}{10}$, where $\lambda_1(x)$ denotes the shortest non-zero vector in $x$ with respect to the supremum norm, which we denote by $\norm{\cdot}$, and let
$E=\{t>0 \ |\ \lambda_1(a_tx)\geq\frac{1}{10}\}$.
Assume that $E$ can be written as a disjoint union of closed intervals
\[ E = [s_1,t_1]\cup[s_2,t_2]\cup\dots\]
where the reals $s_1,s_2,\dots$ and $t_1,t_2,\dots$ are defined inductively by
$s_1=0$, and, for $i\geq 1$,
\[ t_i = \inf\{t>s_i \ |\ \lambda_1(a_tx)\leq\frac{1}{10}\},
\quad\mbox{and}\quad
s_{i+1} = \inf\{s>t_{i} \ |\ \lambda_1(a_tx)\geq\frac{1}{10}\}.\]

Now, for each $i$, choose a non-zero vector $v=\left(\begin{matrix}v_1\\v_2\end{matrix}\right)$ in $x$ such that $\lambda_1(a_{t_i}x)=\|a_{t_i}v\|$.
One readily checks the following
\begin{enumerate}
\item $\forall t\in(t_i,s_{i+1}),\ \lambda_1(a_tx)=\|a_tv\|$
\item $e^{t_i}|v_1|=e^{-s_{i+1}}|v_2|\leq\frac{1}{10}$ and $e^{-t_i}|v_2|=e^{s_{i+1}}|v_1|=\frac{1}{10}$,
\end{enumerate}
using the fact that in dimension 2, if $\|a_tv\|<1$ (and $v$ is primitive), then $\lambda_1(a_tx)=\|a_tv\|$ (see Figure~\ref{atx}).

\begin{figure}
\begin{tikzpicture}
\draw[->] (-1.2,0) -- (2,0);
\draw[->] (0,-1.2) -- (0,2);
\draw (-1,-1) rectangle (1,1);
\draw[thick] (0.3,1) .. controls (0.4,.5) and (.5,0.4) .. (1,0.3);
\draw (0,1) node[anchor=east, fill=white] {\small{$\frac{1}{10}$}};
\draw (1,0) node[anchor=north, fill=white] {\small{$\frac{1}{10}$}};
\draw (.1,-.15) node[anchor=east] {\small{$0$}};
\filldraw[red] (0.3,1) circle (1pt);
\draw (0.3,1) node[anchor=south,red] {\tiny{$a_{t_i}v$}};
\filldraw[blue] (1,0.3) circle (1pt);
\draw (1,0.3) node[anchor=west,blue] {\tiny{$a_{s_{i+1}}v$}};
\end{tikzpicture}
\caption{$a_t v$, for $t\in (t_i,s_{i+1})$}
\label{atx}
\end{figure}
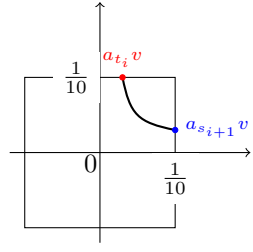

From the fact that $a_{t_i}v$ has norm $\frac{1}{10}$ and makes an angle of at least $\frac{\pi}{4}$ with the first coordinate axis, we see that the translates of the line $\R a_{t_i}v$ by vectors of $a_{t_i}x$ intersect the horizontal axis in a discrete subgroup $\ell_i\Z$, with $\ell_i\in [5,50]$.
Let
\[ B_i = \{ \gamma\in\R \ |\ d(\gamma,e^{-t_i}\ell_i\Z) > 2e^{-t_i}\}.\]
For a grid $y$, we let $\sigma(y) \defn \min\{\|v\| \,;\, v\in y\}$ denote the norm of the shortest vector in $y$.
We claim that
\begin{equation}\label{cla}
\forall\gamma\in B_i,\,\forall t\in[t_i,s_{i+1}],\ 
\sigma(a_t(x+\left(\begin{matrix}\gamma\\ 0\end{matrix}\right)))\geq 1.
\end{equation}
To see this, observe that if $\gamma\in B_i$, then the box
\[ \left(\begin{matrix}e^{t_i}\gamma\\0\end{matrix}\right)+[-1,1]\times [-e^{-2t_i}\frac{v_2}{v_1},e^{-2t_i}\frac{v_2}{v_1}] \]
around $a_{t_i}\left(\begin{matrix}\gamma\\0\end{matrix}\right)$
does not intersect $a_{t_i}x$ (see Figure~\ref{box}).

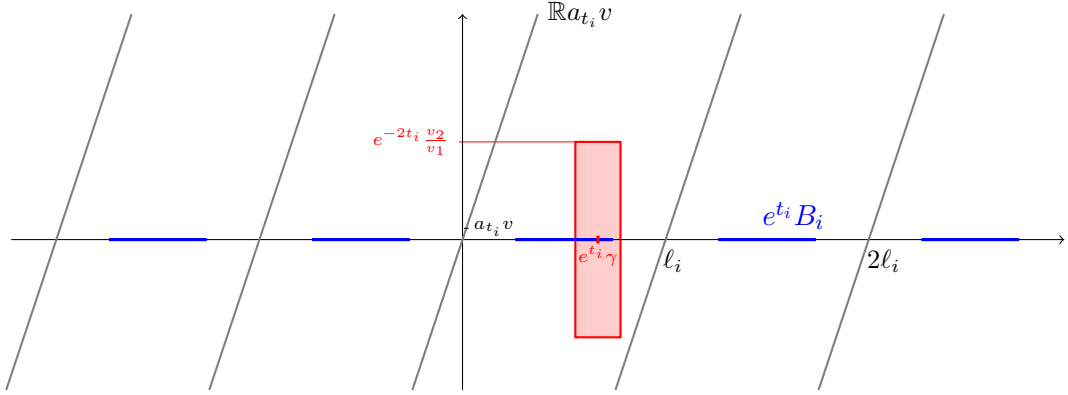
\begin{figure}
\begin{tikzpicture}
\filldraw[red!20!white] (1.5,-1.3) rectangle (2.1,1.3);
\draw[red,thick] (1.5,-1.3) rectangle (2.1,1.3);
\draw (1.8,.05) node[anchor=north,red] {\tiny{$e^{t_i}\gamma$}};
\draw[red,very thin] (-.05,1.3) node[anchor=east] {\tiny{$e^{-2t_i}\frac{v_2}{v_1}$}} -- (1.5,1.3);
\draw[->] (-6,0) -- (8,0);
\draw[->] (0,-2) -- (0,3);
\foreach \x in {-2,-1,...,2}
{
\draw[thick, gray] (-2/3+2.7*\x,-2) -- (1+2.7*\x,3);
\draw[very thick, blue] (2.7*\x+0.7,0) -- (2.7*\x+2,0);
}
\draw (1,3) node[anchor=west] {\small{$\R a_{t_i}v$}};
\draw (2.8,0) node[anchor=north] {\small{$\ell_i$}};
\draw (5.6,0) node[anchor=north] {\small{$2\ell_i$}};
\draw (0.02,0.15) -- (0.08,0.15) node[anchor=west] {\!\tiny{$a_{t_i}v$}};
\draw (4.4,0) node[anchor=south,blue] {$e^{t_i}B_i$};
\draw[red,very thick] (1.8,-.05) -- (1.8,.05);
\end{tikzpicture}
\caption{Translates of $\R a_{t_i}v$, $e^{t_i}B_i$, and the box around $e^{t_i}\gamma$}
\label{box}
\end{figure}

Therefore, if $t\in [t_i,s_{i+1}]$, then the box around the origin
\begin{align*}
[-e^{t-t_i},e^{t-t_i}]\times [-e^{-t_i-t}\frac{v_2}{v_1},e^{-t_i-t}\frac{v_2}{v_1}]
& \supset [-1,1]\times [-e^{-t_i-s_{i+1}}\frac{v_2}{v_1},e^{-t_i-s_{i+1}}\frac{v_2}{v_1}]\\
& = [-1,1]\times [-1,1]
\end{align*}
does not intersect the grid $a_t(x-\left(\begin{matrix}\gamma\\0\end{matrix}\right))$, for any $\gamma\in B_i$.
This proves our claim.

To prove Proposition~\ref{example}, we use the following elementary Hausdorff dimension estimate.

\begin{lem}\label{hdest}
With the above notation, suppose $\lim_{i\to\infty}\frac{t_i}{i}=\infty$.
Then
\[ \dim_H\bigcap_{i\geq 1} B_i = 1.\]
\end{lem}
\begin{proof}
By the mass distribution principle \cite[\S4.2 page 60]{falconer}, it suffices to construct on $B=\bigcap B_i$, for $\eps>0$ arbitrarily small, a probability measure $\mu$ satisfying, for all $x$ and all $r>0$ sufficiently small,
\[ \mu(B(x,r)) \leq r^{1-\eps}.\]
For this, we define $\mu_1$ to be the Lebesgue measure on each interval of $B_1$ included in $[0,1]$, normalized to be a probability measure.
Let $N_1$ be the number of intervals of $B_1$ included in $[0,1]$.
Within bounded multiplicative constants, we have
\[ N_1 \asymp \frac{e^{t_1}}{\ell_1} \asymp e^{t_1}.\]
Then, we let $\mu_2$ be the probability measure compatible with $\mu_1$ (in the sense that the $\mu_2$-mass of a $B_1$-interval is equal to its $\mu_1$-mass) and equal to the appropriately normalized Lebesgue measure on each $B_2$-interval.
The number of $B_2$-intervals inside a $B_1$ interval is
\[ N_2 \asymp \frac{e^{-t_1}\ell_1}{e^{-t_2}\ell_2} \asymp e^{t_2-t_1}.\]
Iterating this procedure, we obtain a sequence of probability measures $\mu_n$ supported on $\bigcap_{i=1}^n B_i$ ; then, we let $\mu$ be a weak-* limit of the sequence $(\mu_n)$.
Note that, by our construction, if $I$ is a $B_i$-interval then for all $n\ge i$, $\mu(I) = \mu_n(I)$.

For $r>0$ sufficiently small, find $i$ such that $e^{-t_{i-1}}\ell_{i-1}>r\geq e^{-t_i}\ell_i$. 
Since $B_i$-intervals are separated by a distance $e^{-t_i}\ell_i$, the number of $B_i$-intervals intersecting $B(x,r)$ is at most $\frac{r}{e^{-t_i}\ell_i}\asymp re^{t_i}$, and the $\mu$-mass of a $B_i$-interval is $\asymp (N_1\dots N_i)^{-1}\leq C^i e^{-t_i}$,  where $C$ is some positive constant independent of $i$, so that
\[ \mu(B(x,r)) \leq re^{t_i}C^ie^{-t_i} \leq rC^i.\]
Using that $r\ll e^{-t_{i-1}}$ and that $\lim \frac{i}{t_{i-1}}=0$, we find that, given any $\eps>0$, for $i$ large enough (i.e. $r$ small enough), $C^i=e^{t_{i-1}\frac{i\log C}{t_{i-1}}}\leq r^{-\eps}$.
Thus, for sufficiently small $r>0$ (depending on $\eps$)
\[ \mu(B(x,r)) \leq r^{1-\eps}.\]
\end{proof}

We can now derive Proposition~\ref{example}.

\begin{proof}[Proof of Proposition~\ref{example}]
Let $\alpha\in [0,1]$ be an irrational number with continued fraction expansion $\alpha=[n_1,n_2,\dots]$ such that $\lim n_i=\infty$, and set
\[ x= \begin{pmatrix} 1 & \alpha\\ 0 & 1\end{pmatrix}\Z^2.\]
The set $E=\{t\geq 0 \ |\ \lambda_1(a_tx)\geq\frac{1}{10}\}$ can be written as a union of disjoint intervals
\[ E = [s_1,t_1]\cup [s_2,t_2] \cup \dots \]
and for some constant $C$, for all $i$, $t_i-s_i\leq C$, and $\lim \frac{t_i}{i}=\infty$.
By Lemma~\ref{hdest}, the set $B=\bigcap_{i\geq 1} B_i$ has Hausdorff dimension 1.
Moreover, by (\ref{cla}), for any $\gamma\in B$, for all $t$ not in any interval $[s_i,t_i]$,
\[ \sigma(a_t(x+\,\left(\begin{matrix}\gamma\\0\end{matrix}\right))) \geq 1.\]
Since the intervals $[s_i,t_i]$ are disjoint and have length at most $C$, we find that for all $\gamma\in B$ and all $t\geq 0$,
\[ \sigma(a_t(x+\left(\begin{matrix}\gamma\\0\end{matrix}\right))) \geq e^{-C}.\]
This shows that for $\cO$ being the $\frac{1}{2}e^{-C}$-ball around the origin with respect to the sup-norm, we have 
that the image of the set $B\times [-1,1]$ in the torus $\bR^2/x$ is contained in $\Ybe(x)$ (note that translating in the stable direction does not affect the asymptotic properties of the $a_t$-orbit), which implies in particular that
\[ \dim_H \Ybe(x) = 2.\]
\end{proof}
\noindent
\tb{Acknowledgments}. We would like to thank Manfred Einsiedler and Elon Lindenstrauss for valuable discussions. 
Seonhee Lim acknowledges the support of the Samsung Science and Technology
Foundation under Project No. SSTF-BA1601-03. 
Nicolas 
de Saxc\'e acknowledges the warm hospitality of the mathematics department at the Technion. 
Uri Shapira acknowledges the support of ISF grant 357/13.


\begin{thebibliography}{BHKV10}

\bibitem[BHKV10]{BHKV10}
Y.~Bugeaud, S.~Harrap, S.~Kristensen, and S.~Velani, \emph{On shrinking targets
  for $\mathbb{Z}^m$ actions on tori}, Mathematika \textbf{56} (2010),
  193--202.

\bibitem[BM92]{BM}
Daniel Berend and William Moran, \emph{The inhomogeneous minimum of binary
  quadratic forms}, Math. Proc. Cambridge Philos. Soc. \textbf{112} (1992),
  no.~1, 7--19. \MR{1162928}

\bibitem[EL10]{EL}
M.~Einsiedler and E.~Lindenstrauss, \emph{Diagonal actions on locally
  homogeneous spaces}, Homogeneous flows, moduli spaces and arithmetic, Clay
  Math. Proc., vol.~10, Amer. Math. Soc., Providence, RI, 2010, pp.~155--241.
  \MR{2648695}

\bibitem[ELMV12]{elmv}
Manfred Einsiedler, Elon Lindenstrauss, Philippe Michel, and Akshay Venkatesh,
  \emph{The distribution of closed geodesics on the modular surface, and
  {D}uke's theorem}, Enseign. Math. (2) \textbf{58} (2012), no.~3-4, 249--313.
  \MR{3058601}

\bibitem[ELW]{ELW}
M.~Einsiedler, E.~Lindenstrauss, and T.~Ward, \emph{Entropy in ergodic theory
  and homogeneous dynamics}, To appear. Preprint available at
  http://www.personal.leeds.ac.uk/\~{}mattbw/entropy.

\bibitem[ET11]{ET}
Manfred Einsiedler and Jimmy Tseng, \emph{Badly approximable systems of affine
  forms, fractals, and {S}chmidt games}, J. Reine Angew. Math. \textbf{660}
  (2011), 83--97. \MR{2855820}

\bibitem[Fal03]{falconer}
Kenneth Falconer, \emph{Fractal geometry}, second ed., John Wiley \& Sons,
  Inc., Hoboken, NJ, 2003, Mathematical foundations and applications.
  \MR{2118797}

\bibitem[KTV06]{KTV}
Simon Kristensen, Rebecca Thorn, and Sanju Velani, \emph{Diophantine
  approximation and badly approximable sets}, Adv. Math. \textbf{203} (2006),
  no.~1, 132--169. \MR{2231044}

\bibitem[KW10]{KWAdvances}
Dmitry Kleinbock and Barak Weiss, \emph{Modified {S}chmidt games and
  {D}iophantine approximation with weights}, Adv. Math. \textbf{223} (2010),
  no.~4, 1276--1298. \MR{2581371}

\bibitem[PV02]{PV}
Andrew Pollington and Sanju Velani, \emph{On simultaneously badly approximable
  numbers}, J. London Math. Soc. (2) \textbf{66} (2002), no.~1, 29--40.
  \MR{1911218}

\bibitem[Sha13]{GDV}
Uri Shapira, \emph{Grids with dense values}, Comment. Math. Helv. \textbf{88}
  (2013), no.~2, 485--506. \MR{3048195}

\end{thebibliography}
\def\cprime{$'$} \def\cprime{$'$} \def\cprime{$'$}
\providecommand{\bysame}{\leavevmode\hbox to3em{\hrulefill}\thinspace}
\providecommand{\MR}{\relax\ifhmode\unskip\space\fi MR }
\providecommand{\MRhref}[2]{%
  \href{http://www.ams.org/mathscinet-getitem?mr=#1}{#2}
}
\providecommand{\href}[2]{#2}

\end{document}